\documentclass[11pt, reqno]{amsart}
\usepackage{hyperref}
\usepackage{cmap}                        
\usepackage[utf8]{inputenc}            
\usepackage[russian,english]{babel}
\usepackage[left=1in,right=1in,top=1in,bottom=1.5in]{geometry} 
\usepackage{amssymb,amsmath, amsthm, amscd,ifthen,xcolor}
\usepackage{graphicx}

\newtheorem*{thm*}{Theorem}
\newcommand{\ff}{{\mathcal F}}
\newcommand{\aaa}{{\mathcal A}}

\newcommand{\G}{{\mathcal G}}

\newcommand{\M}{{\mathcal M}}

\newtheorem*{cla*}{Claim}
\newcommand{\bb}{{\mathcal B}}

\newtheorem{thm}{Theorem}
\newtheorem{gypo}{Conjecture}

\newtheorem{lem}[thm]{Lemma}

\newcommand{\eps}{{\varepsilon}}
\newcommand{\EE}{{\mathfrak E}}

\date{}

\newtheorem{prop}[thm]{Proposition}

\newtheorem{defn}{Definition}

\DeclareMathOperator{\E}{\mathrm E}

\date{}
\title{Rainbow version of the Erd\H os Matching Conjecture via Concentration}
\author{Andrey Kupavskii}
\address{G-SCOP, CNRS, Universit\'e Grenoble-Alpes, France and
Moscow Institute of Physics and Technology, Russia; Email: {\tt kupavskii@ya.ru}.} \thanks{Research of the author is supported by the grant RSF N 21-71-10092}

\begin{document}
\maketitle

\begin{abstract}
  We say that the families $\ff_1,\ldots, \ff_{s+1}$ of $k$-element subsets of $[n]$ are cross-dependent if there are no pairwise disjoint sets $F_1,\ldots, F_{s+1}$, where $F_i\in \ff_i$ for each $i$. The rainbow version of the Erd\H os Matching Conjecture due to Aharoni and Howard and independently to Huang, Loh and Sudakov states that $\min_{i} |\ff_i|\le \max\big\{{n\choose k}-{n-s\choose k}, {(s+1)k-1\choose k}\big\}$ for $n\ge (s+1)k$. In this paper, we prove this conjecture for $n>3e(s+1)k$ and $s>10^7$. One of the main tools in the proof is a concentration inequality due to Frankl and Kupavskii.
\end{abstract}

\section{Introduction}
Let $[n]$ stand for the set $\{1,\ldots, n\}$. For a set $X$, let $2^X$ and ${X\choose k}$ stand for the families of all subsets and all $k$-element subsets of $X$, respectively. A {\it family} is any collection of sets. For a family $\ff$, denote by $\nu(\ff)$ its {\it matching number}, that is, the largest $m$ such that $\ff$ contains $m$ pairwise disjoint sets. One of the most famous open problems in extremal set theory is the Erd\H os Matching Conjecture, which is about the largest size of a family $\ff\subset {[n]\choose k}$ with $\nu(\ff)\le s$. Fix integers $n,k,s$ and consider the following families:
\begin{align*}  \aaa:=& \Big\{F\in {[n]\choose k}: [s]\cap F\ne \emptyset\Big\}, \\
  \bb :=& {[(s+1)k-1]\choose k}.
\end{align*}
The Erd\H os Matching Conjecture states the following.
\begin{gypo}[Erd\H os Matching Conjecture \cite{E}]
  If $n\ge (s+1)k$ and a family $\ff\subset {[n]\choose k}$ satisfies $\nu(\ff)\le s$, then $|\ff|\le \max\{|\aaa|,|\bb|\}$.
\end{gypo}
There are numerous papers devoted to the Erd\H os Matching Conjecture, or EMC for short. The case $s=1$ is the classical Erd\H os--Ko--Rado theorem \cite{EKR} which was the starting point of a large part of ongoing research in extremal set theory.  The cases $k\le 3$ were settled in a series of papers \cite{EG}, \cite{FRR}, \cite{LM}, \cite{F6}.

For larger values of $k$, there has been several improvements on the range of parameters for which the EMC is valid. Erd\H os \cite{E} proved it for $n\ge n_0(k,s)$. Bollob\'as, Daykin and Erd\H os \cite{BDE} established it for $n\ge 2k^3s$. Huang, Loh and Sudakov \cite{HLS} proved the EMC for $n\ge 3k^2s$. Frankl \cite{F4} proved the EMC for $n\ge (2s+1)k-s$. Most recently, Frankl and Kupavskii \cite{FK21} showed that there is an absolute constant $s_0$ such that the EMC is valid for $n\ge \frac 53 sk-\frac 23s$ and $s\ge s_0$.

In the results above, $\aaa$ is the extremal family. An easy, but tedious, computation shows that $|\aaa|>|\bb|$ already for $n\ge (k+1)(s+1)$. For $n$ close to $(s+1)k$, however, $\bb$ is larger. For $n=(s+1)k$ the EMC was implicitly proved by Kleitman \cite{Kl}. This was extended by Frankl who showed that the family $\bb$ is extremal  for all $n\le (s+1)(k+\eps)$, where $\eps=k^{O(-k)}$ \cite{F7}.

The EMC is related to several other exciting problems in combinatorics and probability, and we refer to \cite{aletal}, \cite{FK21} for the survey of different developments.
One of the common and fruitful directions in extremal combinatorics is extending classical extremal statements to a rainbow (multipartite) setting. The following conjecture that generalizes the Erd\H os Matching Conjecture was proposed by Aharoni and Howard \cite{AhH} and independently by Huang, Loh and Sudakov \cite{HLS}.
i\begin{gypo}\label{rainEMC}
  Fix $\ff_1,\ldots, \ff_{s+1}\subset {[n]\choose k}$. If there are no pairwise disjoint $F_1,\ldots, F_{s+1}$, where $F_i\in\ff_i$ for every $i=1,\ldots, s+1$, then
  $$\min_{i\in[s+1]} |\ff_i|\le \max\{|\aaa|,|\bb|\}.$$
\end{gypo}
We call such families {\it cross-dependent}, and refer to this conjecture as the rainbow EMC. The reason for the colorful terminology is that one can think of the families as color classes, and then we are actually looking for a rainbow matching. In the paper \cite{HLS}, the authors proved it for $n\ge 3k^2s$. Keller and Lifshitz in \cite{KLchv} verified the rainbow EMC for $n>f(s)k$ with a very quickly growing $f(s)$ as an application of the junta method. Frankl and Kupavskii \cite{FK22}, using their junta approximation for shifted families, proved it for $n\ge 12 sk \log (e^2s)$. Recently, Gao, Lu, Ma and Yu \cite{GLMY} verified the rainbow EMC for $k= 3$ and $n\ge n_0$ and Lu, Wang and Yu \cite{LWY} verified it for any $k$, $n>2sk$ and $s>s_0(k)$. The last two works use absorption method, which implies that the lower bounds on $n$ (and $s$) are extremely large w.r.t. $k$. Finally, in a very recently published result (which was announced some years ago), Keevash, Lifshitz, Long and Minzer \cite{KLLM} verified the rainbow EMC for $n>Csk$ with some absolute (but very large and unspecified) $C$ as an application of their recent hypercontractivity inequality.

The goal of this article is to prove the rainbow EMC with concrete and reasonable dependencies between the parameters (albeit for relatively large $s$). The proof develops on the ideas from \cite{FK21} and \cite{KK} that used a certain concentration result for intersections of families and matchings, and another goal is to demonstrate an application of this method.

\begin{thm}\label{thmrain}
  There exists $s_0$ such that Conjecture~\ref{rainEMC} is true for any integers $n,s,k$ such that $s>s_0$ and $n>3e(s+1)k$. More precisely, if $\ff_1,\ldots,\ff_{s+1}\subset {[n]\choose k}$ are cross-dependent then
$$\min_{i\in[s+1]} |\ff_i|\le |\aaa|$$
and the inequality is strict unless $\ff_1 = \ldots = \ff_{s+1} = \aaa$.
\end{thm}
{\bf Remark.} We can take $s_0=2\cdot 10^6$ in the present proof. Together with the result of Frankl and Kupavskii \cite{FK22} cited above, we can get that the EMC holds for any $s,k$ and $n>200sk$. This can clearly be significantly improved using the present method but  requires more tedious calculations that we decided to avoid. We have managed to reduce the present bound to $s_0>700$ in the  assumption $n>30sk$. Altogether, this  would imply the result for any $s$ and  for $n>100sk$ and any $s,k$. Anything significantly better seems to be difficult to get using the present method.
\vskip+0.1cm

We also note that there were several other developments related to the EMC. Let us mention the following two. First, a rainbow version of the EMC for multipartite hypergraphs was proved by Kiselev and Kupavskii in \cite{KK} for all $s\ge 500$. Second, Frankl and Kupavskii \cite{FK23} studied a common generalization of the EMC and the Complete $t$-Intersection Theorem.

\section{Sketch of the proof}
This proof follows the strategy of the proof of the EMC due to Frankl \cite{F4} and its extension due to Frankl and Kupavskii \cite{FK21}. First, let us recall the strategy of Frankl. The first step of the proof is to reduce the problem to shifted families and decompose the family $\ff\subset {[n]\choose k}$ into parts $$\ff(X,[s+1]):=\{F\setminus X: F\cap [s+1] = X, F\in \ff\}.$$
 An easy calculation shows that, in order to verify the EMC, it is sufficient to check that $$|\ff(\emptyset, [s+1])|+\sum_{i=1}^{s+1}|\ff(\{i\},[s+1])|\le s{n-s-1\choose k-1}.$$ Next, Frankl shows that $|\ff(\emptyset, [s+1])|\le s|\ff(\{s+1\},[s+1])|$ using the fact that $\nu(\ff)\le s$. Finally, the last step is an averaging argument that verifies  $$(s+1)|\ff(\{s+1\}, [s+1])|+\sum_{i=1}^{s}|\ff(\{i\},[s+1])|\le s{n-s-1\choose k-1}.$$ The argument goes by verifying the corresponding inequality for a random almost-perfect matching $\mathcal M$. The proof crucially relies on the fact that $\ff(\{s+1\},[s+1])\subset \ff(\{i\},[s+1])$ for any $i\in[s]$.

One important ingredient added by Frankl and Kupavskii at the third step of the proof is that, for most matchings $\mathcal M$, the intersection of a family $\ff(\{s+1\},[s+1])$ with $\mathcal M$ has roughly the same density as the family $\ff(\{s+1\},[s+1])$ itself. (See Theorem~\ref{thmrainconcen} below.)

There were two obstacles to extending the approach described above to the rainbow EMC, which were previously believed to be unsurpassable (cf. e.g. \cite{FK22}). First, since none of the families $\ff_1,\ldots, \ff_{s+1}$ in the rainbow EMC has to have matching number at most $s$, there was no known analogue of the inequality $|\ff(\emptyset, [s+1])|\le s|\ff(\{s+1\},[s+1])|$. We managed to find the right property to work with and got such an analogue, essentially contained in Lemma~\ref{lemrain4}. Second, and probably most importantly, since we look at rainbow matchings, the property $\ff(\{s+1\},[s+1])\subset \ff(\{i\},[s+1])$ is of little use for our situation, and there is no analogue for the case when the families from the inclusion are subfamilies of different $\ff_i,\ff_j$.

Our proof follows the rough outline of the proof by Frankl and by Frankl and Kupavskii. We, however, need several additional ingredients.  First, the case when all families are very close to $\aaa$ is dealt with separately, in the same vein as it was dealt with in the paper \cite{FK22}. In case when some of the families are far from $\aaa$, we use the same decomposition of $\ff_i$ into $\ff_i(X,[s+1])$ and aim at the same inequality for $\ff_i(\{j\},[s+1])$ as before. We make strong use of the aforementioned concentration result (Theorem~\ref{thmrainconcen}), in a way that it essentially allows us to conclude that in almost all random matchings the proportions of each of the considered family is correct, and so we can connect the densities of these families with the number of sets from a random matching that lie in these families. We then identify a group of families that must satisfy the analogue of the inequality $|\ff(\emptyset, [s+1])|\le s|\ff(\{s+1\},[s+1])|$, and aim to prove that one of them must be small. Recall that both \cite{F4} and \cite{FK21} also reduced the situation to the analysis of what happens on a random matching. Our analysis, however,  is different and inspired by \cite{KK}. It uses the following simple fact that follows from the K\"onig--Hall theorem: if for some matching $\mathcal M$ we have $|\ff_i\cap\mathcal M|\ge i$ for each $i=1,\ldots, s+1$ then $\ff_1,\ldots, \ff_{s+1}$ contain a rainbow matching.

\section{Proof of Theorem~\ref{thmrain}}

Let $\ff_1,\ldots, \ff_{s+1}\subset {n\choose k}$ be cross-dependent and such that $|\ff_i|\ge|\aaa|$ for all $i$. Recall that $\ff$ is called {\it shifted} if whenever $A\in \ff$ and $B$ is obtained from $A$ by replacing some larger elements with smaller, then $B\in \ff$. It is  standard that we can w.l.o.g. assume that each of the families is shifted (see, e.g., \cite{F4}). We make this assumption throughout the proof.

The first step of the proof, which is not strictly necessary but convenient, is to reduce the case of general $n$ to the case of the smallest $n$  satisfying the requirements.

The following proposition implies that it is sufficient to prove Theorem~\ref{thmrain} for $n = \lceil 3e(s+1)k\rceil$.

\begin{prop}
Assume that $n\ge (s+1)k$. If for any cross-dependent families $\mathcal G_1,\ldots, \mathcal G_{s+1}\subset {[n]\choose k}$ we have $\min_i|\mathcal G_i|\le {n\choose k}-{n-s\choose k}$, then for any cross-dependent families $\ff_1,\ldots, \ff_{s+1}\subset {[n+1]\choose k}$ we have $\min_i|\ff_i|\le {n+1\choose k}-{n+1-s\choose k}$.
\end{prop}

\begin{proof} Consider an  $(s+1)$-tuple of cross-dependent families $\ff_1,\ldots, \ff_{s+1}\subset {[n+1]\choose k}$ such that $\min_{i}|\ff_i|$ is maximal and each family is inclusion-maximal. We may w.l.o.g. assume that $\ff_i$ are shifted. Let us put $\G_i:=\{A\in \ff_i:n+1\notin A\}$ and $\G_i':=\{A\setminus \{n+1\}: A\in \ff_i, n+1\in A\}$. The families $\G_1,\ldots, \G_{s+1}$ are cross-dependent, and, by our assumption, $\min_{i}|\G_i|\le {n\choose k}-{n-s\choose k}$. 
Let us assume that, say,
$$|\G_1|\le  {n\choose k}-{n-s\choose k}.$$
We claim that $|\G_1'|\le {n\choose k-1}-{n-s\choose k-1}.$

For a family $\mathcal H\subset {[n]\choose k}$ let us put $\overline \partial \mathcal H$ to be {\it upper shadow} of $\mathcal H$, i.e., the family of all $(k+1)$-sets that contain at least one of the sets from $\mathcal H$.

It is a standard application of the properties of shifting that $\G_1',\ldots, \G_{s+1}'$ are cross-dependent. Therefore, the families $\G_1'\cup\overline\partial \G_1',\ldots, \G_{s+1}'\cup \overline\partial \G_{s+1}'$ are cross-dependent as well. This and the inclusion-maximality of the families $\ff_1,\ldots, \ff_{s+1}$ implies that $\overline \partial \G_1'\subset \G_1$.

Assume that $|\G_1'|>{n\choose k-1}-{n-s\choose k-1}$. The Kruskal-Katona theorem \cite{Kr, Ka} stated in terms of the upper shadow for a ground set of fixed size implies that the upper shadow is minimized if the family $\G_1'$ consists of lexicographically first $|\G_1'|$ sets.\footnote{To see this, replace each set in the family by its complement. Then the upper shadow corresponds to the normal shadow, and lexicographic order on the initial family corresponds to the colexicographic order on the resulting family.} In particular, the inequality on $|\G_1'|$ implies that $|\overline\partial\G_1'|>{n\choose k}-{n-s\choose k}$, a contradiction with our assumption on $|\G_1|$. Therefore, $|\G_1'|\le{n\choose k-1}-{n-s\choose k-1}$ and $|\ff_1| = |\G_1|+|\G_1'| \le {n+1\choose k}-{n+1-s\choose k}$, which concludes the proof.
\end{proof}

In what follows, we assume that $n = \lceil 3e(s+1)k\rceil$.  Put $n' := n-s-1$ and $X := [s+2,n]$. Let us put $t:=\big\lfloor \frac {n'}{k} \big\rfloor$ and note that $t<3e(s+1)$. Note also that \begin{equation}\label{rain-1} t>7(s+1)+2 \ \ \ \ \ \ \text{for } s>10.\end{equation}
In the proof we often omit floors and ceiling signs whenever they do not affect the calculations.

\subsection{The families that are all close to $\aaa$} The next lemma deals with the case when all families are very close to $\aaa$. Recall the following notation for  two sets $S,Y$ such that $S\subset Y\subset [n]$ and a family $\G\subset 2^{[n]}$:
\begin{align}
\mathcal \G(S,Y) :=& \{A\setminus S: A\in \G, A\cap Y = S\}\end{align}
For shorthand, we also use the following notation for $1\le i,j\le s+1$.
\begin{align}\G(j) :=& \{A\setminus \{j\}: A\in \G, A\cap [s+1] = \{j\}\},\\
\G(\emptyset):=& \{A\in \G: A\cap [s+1] = \emptyset\}.\end{align}

\begin{lem}\label{lemrain}
   We have $\min_{i\in [s+1]}|\ff_i|<{n\choose k}-{n-s\choose k}$ if \begin{equation}\label{rain10}|\ff_i(\{s+1\})|< s^{-4}{n-s-1\choose k-1} \ \ \ \ \text{for all }i\in[s+1]\end{equation}
and at least one of $\ff_i$ does not coincide with $\aaa$.
\end{lem}
In what follows, for a family $\G\subset {[n]\choose \ell}$ we denote by $\partial \G$ its (lower) shadow, i.e., $\partial\G = \cup_{X\in \G}{X\choose \ell-1}$.
\begin{proof}
  Recall that $\ff_1,\ldots, \ff_{s+1}$ are shifted. This implies that $\ff_i(s+1)\supset \partial\ff_i(\emptyset)$. Indeed, we can replace any element of $A\in \ff_i(\emptyset)$ by $s+1$ and get a set from the family. At the same time, a local LYM argument implies that $$t|\partial\ff_i(\emptyset)| \ge |\ff_i(\emptyset)|.$$
   Let us sketch the argument. Consider a regular bipartite graph with parts ${X\choose k}$ and ${X\choose k-1}$, where edges connect pairs of sets in which one contains the other. Then the degree of each vertex in the first and second part is $k$ and $n'-k+1$, respectively. Note that $\frac {n'-k+1}k<t$. Now consider the family $\ff_i(\emptyset)$ as a subset of the part ${X\choose k}$. Its neighborhood in the graph is precisely $\partial \ff_i(\emptyset)$. A simple double counting implies that $(n'-k+1)|\partial \ff_i(\emptyset)|\ge k|\ff_i(\emptyset)|,$ which, in turn, implies the claimed bound.

  Combining the displayed inequality with \eqref{rain10}, we conclude that for any $i\in[s+1]$ we have
  $$|\ff_i(s+1)|+|\ff_i(\emptyset)|\le (t+1)s^{-4}{n-s-1\choose k-1}.$$
   Next, we adapt the argument from \cite[Theorem~21, p.15]{FK21} for this case.

  Assume that, among $i\in [s+1]$, the density $\alpha_i:=|\ff_i(\emptyset,[s])|/{n-s\choose k}$ is the largest for $i=1$ and put $\beta_i^l:=|\ff_i(\{l\},[s])|/{n-s\choose k-1}$ for each $i\in[2,s+1]$ and $l\in [s]$. We may assume that $\alpha_1>0$, otherwise $\ff_i\subset \aaa$ for all $i\in [s+1]$. For a finite set $Y$, a family $\mathcal G\subset{Y\choose k}$ and an integer $u$, $k\le u\le |Y|$, let $\bar\partial^u\mathcal G$ be the collection of all sets in ${Y\choose u}$ that contain at least one set from $\mathcal G$. The following analytic corollary of the Kruskal--Katona  theorem \cite{Kr,Ka} was proved by Bollob\'as and Thomason \cite{BT}:
$$\Big(|\bar\partial^u\mathcal G|/{|Y|\choose u}\Big)^{|Y|-k}\ge \Big(|\mathcal G|/{|Y|\choose k}\Big)^{|Y|-u}.$$
We apply it to $\ff_1(\emptyset, [s])$ with $Y =[s+1,n]$ and $u=2(n-s+k)/3$ and conclude that $$\alpha'_1:=\frac{|\bar\partial^{2(n-s+k)/3} \ff_1(\emptyset, [s])|}{{n-s\choose 2(n-s+k)/3}}\ge \alpha_1^{1/3}.$$

At the same time, for any bijection $\pi:[s]\to [2,s+1]$, the families $\ff_{\pi(j)}(\{j\},[s]),$ $j\in[s]$, and $\bar\partial^{2(n-s+k)/3}\ff_1(\emptyset,[s])$ are cross-dependent. Since $2(n-s+k)/3+s(k-1)<n-s-1$, we may take a random ordered matching consisting of $s$ sets  $M_1,\ldots, M_{s}$ of size $k-1$ and one set $M_{s+1}$ of size $2(n-s+k)/3$. Due to cross-dependency of the aforementioned families, in any such matching there are at most $s$ indices $j\in [s+1]$ such that:  $M_j\in \ff_{\pi(j)}(\{j\},[s])$  for $j\in[s]$, or   $M_j\in \bar\partial^{2(n-s+k)/3}\ff_1(\emptyset,[s])$ for $j=s+1$. Computing the expectation of the number of such indices, we get
$$\alpha'_1+\sum_{j=1}^{s}\beta_{\pi(j)}^j\le s.$$
Therefore, there exists $i\in [2,s+1]$, such that $\sum_{j=1}^{s}\beta_i^j\le s-\alpha'_1$. Comparing $\ff_i$ to $\aaa$, we conclude that at least $\alpha_1'{n-s\choose k-1}$ sets intersecting $[s]$ in a single element are missing from $\ff_i$, and, at the same time, $\alpha_i{n-s\choose k}$ sets in $\ff_i$ do not intersect $[s]$ and thus are not present in $\aaa$. This implies that
\begin{align*}|\ff_i|\le& {n\choose k}-{n-s\choose k}-\alpha'_1{n-s\choose k-1}+\alpha_i{n-s\choose k}\\
\le& {n\choose k}-{n-s\choose k}-\alpha^{1/3}_1{n-s\choose k-1}+\alpha_1{n-s\choose k}\\
\le& {n\choose k}-{n-s\choose k}-\big(\alpha^{1/3}_1-\frac{n-s-k+1}k\alpha_1\big){n-s\choose k-1}\\
<& {n\choose k}-{n-s\choose k},\end{align*}
where the last inequality is due to our choice of parameters. Indeed, we need to verify that $\frac{n-s-k+1}k\alpha^{2/3}_1<1$. Recall that $\alpha_1<(t+1)s^{-4}$. We have $\frac{n-s-k+1}k\alpha^{2/3}_1<(t+1)\alpha_1^{2/3}< (t+1)^{5/3}s^{-8/3}$. The latter expression is smaller than $1$ provided $s^{8/5}>t+1$, which holds for  $s\ge 50$.
\end{proof}

In view of Lemma~\ref{lemrain}, in what follows we may assume that
\begin{equation}\label{rain101}|\ff_i(s+1)|\ge s^{-4}{n-s-1\choose k-1} \ \ \ \ \text{for at least one  }i\in[s+1]\end{equation}

\subsection{Shadows of cross-dependent families: an analogue of $s|\partial \ff|\ge |\ff|$}
In this section we shall use that $\ff_1,\ldots, \ff_{s+1}\subset {[n]\choose k}$ are cross-dependent and shifted, as well as that $|\ff_i|\ge |\aaa|$.

We will need the following lemma that in our context replaces the result of Frankl stating that $s|\partial\G|\ge|\G|$ for any $\G\subset {[n]\choose k}$ with no $s+1$ pairwise disjoint sets.
\begin{lem}\label{lemrain2}
  There is a set $U\subset[s+1]$ of indices, $|U|= 2(s+1)/3$, such that for each $i\in U$ we have
   \begin{equation}\label{betterf} |\ff_{i}(\emptyset)|\le (3s+2) |\ff_i(s+1)|.\end{equation}
\end{lem}
\begin{proof}
Recall that $\ff_1,\ldots, \ff_{s+1}$ are cross-dependent and shifted. A well-known corollary of these two properties is that for any $F_1\in \ff_1,\ldots, F_{s+1}\in\ff_{s+1}$ there exists $\ell$ such that $\sum_{i=1}^{s+1} |F_i\cap [\ell]|\ge \ell+1$ (see, e.g. \cite{Fra3}).  Let us denote by $\beta_i$ the largest rational number such that for any $F_i\in\ff_i$ there exists $\ell$ such that $|F_i\cap [\ell]|\ge \beta_i \ell$. Note that the property we mentioned implies that
\begin{equation}\label{rain100}\sum_{i=1}^{s+1}\beta_i >1.\end{equation}
We shall need the following technical lemma.
\begin{lem}\label{lemrain3}
If  $\beta_i> \frac 4{3(s+1)}$ for some $i\in [s+1]$ then $|\ff_i|<|\aaa|$.
\end{lem}
\begin{proof}[Proof of Lemma~\ref{lemrain3}] Assume that for some $i\in [s+1]$ we have $\beta_i>\frac 4{3(s+1)}.$ As usual, we assume that $\frac 34 (s+1)$ is an integer (this does not affect the validity of the argument below, but makes notation cleaner). The set $G :=\{\frac 34(s+1), 2\cdot\frac 34(s+1),\ldots, k\cdot\frac 34(s+1)\}\notin \ff_i$, because for $G$ there is no $\ell$ such that $|G\cap [\ell]|> \frac 4{3(s+1)}\ell$. This is sufficient to check only for those $\ell$ that satisfy $\ell\in G$, and for such $\ell$ it is straightforward. Therefore, shiftedness implies that for any $A\in\ff_i$ there is a positive integer $p$ such that $|A\cap \big[\frac 34(s+1)p-1\big]| \ge p$. Let us put $s':=\frac 34(s+1)$. By taking the largest such $p$ for a set $A$, we conclude that for any $A\in\ff_i$ there is a positive integer $p$ such that $|F_i\cap [ps'-1]|= p$. Therefore,
$$|\ff_i|\le \sum_{p=1}^{k}{s'p-1\choose p}{n-s'p+1\choose k-p}.$$
Using (in the penultimate inequality below) that $\frac{(k-\delta)s'}{n-\delta s'}<\frac {ks'}n$ for any $0<\delta\le k$ and $n>ks'$, we have the following for any $2\le p\le k$:
\begin{small}\begin{align*}\frac{{s'p-1\choose p}{n-s'p+1\choose k-p}}{{s'(p-1)-1\choose (p-1)}{n-s'(p-1)+1\choose k-(p-1)}}
=& \frac{(p-1)!(k-(p-1))!}{p!(k-p)!}\cdot \frac{(s'p-1)!(s'(p-1)-p)!}{(s'(p-1)-1)!(s'p-p-1)!}\cdot \\ &\cdot \frac{(n-s'p+1)!(n-s'(p-1)+1-(k-(p-1)))!}{(n-s'(p-1)+1)!(n-s'p+1-(k-p))!}\\
=&\frac{k-p+1}{p}\cdot\frac{\prod_{j=0}^{s'-1}(s'p-j-1)} {\prod_{j=1}^{s'-1}(s'p-p-j)}\cdot \frac{\prod_{j=1}^{s'-1}(n-s'(p-1)-(k-(p-1))-j+2)}{\prod_{j=0}^{s'-1}(n-s'(p-1)-j+1)}\\
=& \frac{(k-p+1)(s'p-1)}{p(n-s'(p-1)+1)}\cdot \prod_{j=1}^{s'-1}\Big(1+\frac{p-1}{s'p-p-j}\Big)\Big(1-\frac{k-(p-1)-1}{n-s'(p-1)-j+1}\Big)\\
\le& \frac{(k-p+1)s'}{n-s'(p-1)} \cdot \prod_{j=1}^{s'-1}\Big(1+\frac{p-1}{s'p-p-j}\Big)\\
\le& \frac{(k-p+1)s'}{n-s'(p-1)} \cdot \Big(1+\frac{p-1}{s'p-p-s'+1}\Big)^{s'-1}\\
\le& \frac{ ks'}{n} \cdot \Big(1+\frac{1}{s'-1}\Big)^{s'-1}\\
\le& \frac{ eks'}{n}.
\end{align*}
\end{small}
The last expression is at most $\frac 14$, given that $n\ge 3ek(s+1) = 4eks'$.
With this inequality, we deduce that
$$|\ff_i|\le {s'-1\choose 1}{n-s'+1\choose k-1}\sum_{p=1}^{\infty}4^{1-p}= \frac 43(s'-1){n-s'+1\choose k-1} < s{n-s'+1\choose k-1}.$$
Note that $|\aaa| = \sum_{i=1}^s {n-i\choose k-1}$. Note that the following holds for any integer  $i>0$ and $k\ge 1$: $${n-s'-i-1\choose k-1}+{n-s'+i+1\choose k-1}\ge {n-s'-i\choose k-1}+{n-s'+i\choose k-1}.$$ Using this inequality and assuming that $s$ is odd, we get \begin{multline*}s{n-s'+1\choose k-1}\le{n-s'+1\choose k-1}+\sum_{i=1}^{(s-1)/2}\Big({n-s'+1+i\choose k-1}+{n-s'+1-i\choose k-1}\Big)  \\ =\sum_{i=s'-1-(s-1)/2}^{s'-1+(s-1)/2}{n-i\choose k-1}< \sum_{i=1}^s {n-i\choose k-1}.\end{multline*}
A similar calculation gives the same conclusion for $s$ even. Thus, we get that $|\ff_i|<|\aaa|$.
\end{proof}

In view of the lemma above, we may assume that $\beta_i\le \frac 4{3(s+1)}$ for all $i\in[s+1]$. Let $W\subset [s+1]$ be the set of all indices $i$ such that $\beta_i>\frac 1{3(s+1)}$. 
Using \eqref{rain100} and our assumptions on $\beta_i$, we have $$1<\sum_{i=1}^{s+1}\beta_i= \sum_{i\in W}\beta_i+\sum_{i\in [s+1]\setminus W}\beta_i\le |W|\cdot \frac 4{3(s+1)}+(s+1-|W|)\frac 1{3(s+1)}= \frac 13+\frac {|W|}{s+1},$$ which implies that $|W|>\frac 23(s+1)$.

Arguing similarly to how we argued in the beginning of the proof of Lemma~\ref{lemrain3}, we get that the set $G':=\{3(s+1),6(s+1),\ldots, 3k(s+1)\}$ is not in $\ff_i$ for $i\in W$, and therefore for any $F\in \ff_i$ there exists a positive integer $\ell$  such that $|F\cap [3(s+1)\ell-1]|\ge \ell$. We shall prove the following lemma.
\begin{lem}\label{lemrain4}
Let $n\ge 3(s+1)k-1$. Fix a family $\ff\subset {[n]\choose k}$. Assume that for any $A\in \ff$ there is a positive integer $\ell$ such that  $|A\cap [3(s+1)\ell-1]|\ge \ell$. Then
$$(3s+2)|\partial\ff|\ge|\ff|.$$
\end{lem}
Lemma~\ref{lemrain4} is a corollary of a more general result due to Peter Frankl (Theorem~\ref{thmfrfr}), which we shall state in Section~\ref{sec4}. We will use the same technique as that for the proof of Lemma~\ref{lemrain4} in order to give a simpler proof of that result. Since the result of Frankl requires some extra effort to grasp and its new proof is of independent interest, we decided to present it in concluding remarks.

Since $\ff_i(\emptyset)\subset \ff_i$, the conditions of Lemma~\ref{lemrain4} hold for $\ff_i(\emptyset)$, and thus we conclude that for any $i\in W$ we have $|\ff_{i}(\emptyset)|\le (3s+2) |\partial\ff_i(\emptyset)|\le (3s+2) |\ff_i(s+1)|.$ We are only left to take a subset $U\subset W$ of size $2(s+1)/3$. 
This finishes the proof of Lemma~\ref{lemrain2} modulo Lemma~\ref{lemrain4}.

\begin{proof}[Proof of Lemma~\ref{lemrain4}] The proof is by induction on $n,k$. More precisely, we use the statement for pairs $(n-1,k)$ and $(n-1,k-1)$ in order to deduce it for the pair $(n,k)$. The base cases are $k=1$ and $n = 3(s+1)k-1$. For $k=1$ we simply use that $\ff\subset \{\{1\},\ldots, \{3s+2\}\}$. In the case $n= 3(s+1)k-1$ we do not have any restrictions on $\ff$ and simply use the double counting bound analogous to \eqref{rain01}.

Let us justify the induction step. Let us put $\ff':=\{F\setminus \{n\}: n\in F, F\in \ff\}$ and $\ff'':=\{F: n\notin F, F\in \ff\}$.  We have $|\ff| = |\ff'|+|\ff''|$. Moreover, $\partial \ff\supset \G\cup \partial \ff''$, where $\G:= \{\{n\}\cup A: A\in \partial \ff'\}$ and, clearly, $\G\cap \partial\ff'' = \emptyset$ (all sets from the first family contain $n$, while the sets from the second family do not). Note that the inductive hypothesis applies to both $\ff'$ and $\ff''$. In the first case, this is due to the fact that $n\ge 3(s+1)k$ and thus for any $A\in \ff'$ the $\ell$-condition on the set $\{n\}\cup A$ must be satisfied for some $\ell\le k-1$. In the second case, this is simply because $\ff''\subset \ff$. By induction, $(3s+2)|\G| = (3s+2)|\partial\ff'|\ge |\ff'|$ and $(3s+2)|\partial \ff''|\ge |\ff''|$. Therefore, $(3s+2)|\partial\ff|\ge (3s+2)|\G|+(3s+2)|\partial\ff''|\ge |\ff'|+|\ff''| =|\ff|$.
\end{proof}
This completes the proof of Lemma~\ref{lemrain2}.
\end{proof}

\subsection{The body of the proof}  
The shiftedness of $\ff_i$ implies \begin{equation}\label{rain0}\partial (\ff_i(\emptyset))\subset \ff_i(s+1)\subset \ff_i(s)\subset\ldots\subset \ff_i(1).\end{equation}

Note that $\aaa(S,[s+1]) = {X\choose k-|S|}$ for any $S$ such that $|S|\ge 2$, and thus $\ff_i(S,[s+1])\subset \aaa(S,[s+1])$ for every such $S$. Also, note that $\aaa(i) = {X\choose k-1}$ if $i\le s$, and that $\aaa(s+1) = \emptyset$ and $\aaa(\emptyset) = \emptyset$. Thus, in order to prove the theorem, it is sufficient to check that for some $i\in[s+1]$ we have
\begin{equation}\label{rain000}
  |\ff_i(\emptyset)|+\sum_{j=1}^{s+1}|\ff_i(j)|\le |\aaa(\emptyset)|+\sum_{j=1}^{s+1}|\aaa(j)| = s{n'\choose k-1}.
\end{equation}

We have already verified in the proof of Lemma~\ref{lemrain} that \begin{equation}\label{rain01}|\ff_{i}(\emptyset)|\le t|\partial \ff_i(\emptyset)|\le t|\ff_i(s+1)|.\end{equation}

If $|\ff_i(\emptyset)|\le \tau_i |\ff_i(s+1)|$, then the same analysis as that leading to \eqref{rain000} and the fact that  $|\ff_i|\ge |\aaa|$ imply
$$(\tau_i+1)|\ff_i(s+1)|+\sum_{j=1}^s|\ff_i(j)|\ge s{n'\choose k-1}.$$
In view of \eqref{rain01} and Lemma~\ref{lemrain2} we can take $\tau_i:= 3s+2$ for a set $U$ of $2(s+1)/3$ indices $i\in [s+1]$ and $\tau_i:= t$ for other $i$.

Denote $$\alpha_{i,j}:= \frac {|\ff_i(j)|}{{n'\choose k-1}}$$
and take a uniformly random matching $\M$ of size $t$ from ${[s+2,n]\choose k-1}$. Rewriting the penultimate displayed inequality in terms of the expected intersections with $\M$, for any $i\in [s+1]$ we have
\begin{equation}\label{rain5}(\tau_i+1)t\alpha_{i,s+1}+\sum_{j=1}^st\alpha_{i,j}\ge st.\end{equation}



Recall the following result from \cite{FK21} (stated with the parameters that are convenient for us).
\begin{thm}[Frankl, Kupavskii \cite{FK21}]\label{thmrainconcen} Suppose that $n',k,t$ are positive integers and $n'\ge (k-1)t$. Let $\G\subset {[n']\choose k-1}$ be a family and $\alpha:=|\G|/{n'\choose k-1}$. Let $\eta$ be the random variable equal to the size of the intersection of $\G$ with a $t$-matching $\M$ of $(k-1)$-sets,  chosen uniformly at random. Then $\E[\eta]=\alpha t$ and, for any positive $\beta$, we have \begin{equation}\label{eqconcen} \Pr\big[|\eta - \alpha t|\ge 2\beta \sqrt t\big]\le 2 e^{-\beta^2/2}.\end{equation}
\end{thm}

Using \eqref{eqconcen} with $\beta = 5\sqrt{\log s}$ for each $\ff_i(j)$ and applying the union bound we get that  \begin{equation}\label{rain21}\big| |\ff_i(j)\cap \M| - \alpha_{i,j} t\big|\le 10 \sqrt {t\log s}:=\gamma \end{equation}
for every $1\le i,j\le s+1$ with probability at least $1-2(s+1)^2 \cdot e^{-12.5\log s}>1-s^{-10}$, where the last inequality is valid for any $s>20$. 

We remark here that, in view of the inequality $t<3e(s+1)$, we have
\begin{equation}\label{rain210}
  \gamma+1<\frac{s}{12}
\end{equation}
for any $s\ge 2\cdot 10^6$. This is essentially where the value of $s_0$ comes from.

We also note that \eqref{eqconcen} can be reproved in a stronger form, where essentially $\sqrt {t\log s}$ on the left hand side can be replaced by $\sqrt {\alpha t \log s}$ (cf. \cite[Theorem 6]{KK} for such a statement in a related setting). Also, the inequality \eqref{rain210} can be weakened for larger $n$. I.e., for $n>40sk$ the bound $\gamma+1<\frac s4$ is sufficient. Combined together, this will significantly improve the bounds on $s_0$ (to about $s_0 = 10^3$), but lead to a more technical proof, so we avoid it.



The following technical lemma exploits \eqref{rain5}. 

\begin{lem}\label{lemrain5}
  (i) For any $i\in[s+1]$ we have $t\alpha_{i,(s+1)/3}\ge s+1+\gamma$ or $t\alpha_{i,s+1}\ge \frac {s+1}3+\gamma.$ \\
  (ii) For any $i\in U$ and $j\ge (s+1)/6$ we have $t\alpha_{i,s+1-j}\ge j+1+\gamma$.
\end{lem}
\begin{proof}
(i) Assume that neither of the inequalities holds for some $i$. Then we can bound the left hand side of \eqref{rain5} as $$(t+1)\big(\frac {s+1}3+\gamma\big)+\frac{s} 3t+\frac{2s}3(s+1+\gamma)\le (\gamma+1)(s+t+1)+s\frac {2s+2t}3\le (\gamma+1)(s+t+1)+\frac {5st}6,$$
where the last inequality holds provided $t\ge 6s$ (cf. \eqref{rain-1}). The last expression is smaller than $st$ provided $\gamma+1<\frac {st}{6(s+t+1)}$. The right hand side is at least $\frac s7$ for $t\ge 6s$.  At the same time, we have $\gamma+1\le \frac{s+1}{12}$ by \eqref{rain210}, which shows the validity of the inequality.

(ii) This part has a similar proof, except we use \eqref{rain5} with $\tau_i = 3s+2$ in this case. If one of the inequalities fail then we can bound the left hand side of \eqref{rain5} by
$$(3s+3+j)\big(j+1+\gamma\big)+(s-j)t \le  st - jt+(3s+j)j+(3s+3+4j)(1+\gamma)\le st-\frac {jt}3+(3s+3+4j)(1+\gamma),$$
where the last inequality holds for $\frac 23t\ge 4s\ge 3s+j$. If $1+\gamma<\frac{jt}{3(3s+3+4j)}$ then  the last expression is smaller than $st$, a contradiction. An easy calculation shows that $\frac{jt}{3(3s+3+4j)}>\frac{s+1}{11}$ for $j\ge \frac{s+1}6$ and $t\ge 6s+6$. Again, using \eqref{rain210}, we have the desired inequality.
\end{proof}

Recall that \eqref{rain21} holds for every $1\le i,j\le s+1$ with probability at least $1-s^{-10}$. Let us denote this event $\EE_1$. 

Let us also denote $\EE_2$ the event that $\ff_i(s+1)\cap \M\ne \emptyset $ for at least one $i$. Using \eqref{rain10}, there is $i\in[s+1]$ such that $\alpha_{i,s+1}\ge s^{-4}$, and thus $\EE_2$ happens (even with that particular $i$) with probability at least $\frac 1t\E[\ff_i(s+1)\cap\M] = \alpha_{i,s+1}\ge s^{-4}$. 
\subsection{Rearranging the families and failing to find a rainbow matching} For this subsection, $\M$ is some fixed matching as above.
We rearrange the families according to the following rules. The first two rules are as follows. The integer $s_1$ below satisfies $0\le s_1\le 2(s+1)/3$. We will later show that $s_1>0$.
\begin{itemize}
  \item[(R1)] For each $i\in [2(s+1)/3]$ the inequality \eqref{betterf} holds. This is possible due to Lemma~\ref{lemrain2}. Note that for each such $i$ the conclusion of Lemma~\ref{lemrain5} (ii) holds.
  \item[(R2)] For all $i\in [s_1]$ we have $t\alpha_{i,s+1}\le \frac {s+1}6+\gamma$, and for all $i\in [s_1+1,2(s+1)/3]$ we have $t\alpha_{i,s+1}>\frac{s+1}6+\gamma$. 
\end{itemize}

The next rules are in the assumption that $\EE_1\cap \EE_2$ holds for $\M$. Note that $\EE_1\cap \EE_2$ has positive probability (which is at least $s^{-4}-s^{-10}$).

Lemma~\ref{lemrain5} (i) implies that, whenever $\EE_1$ holds, for any $i\in [s+1]$ we have at least one of the following: $$\big|\ff_i\big((s+1)/3\big)\cap \M\big|\ge s+1\ \ \ \text{ or } \ \ \  |\ff_i(s+1)\cap \M|\ge \frac{s+1}3.$$

Split the set $\big[s_1+1,s+1\big]$ into two disjoint parts  $W_1$ and $W_2$, where \begin{itemize}
                   \item[(a)] $[s_1+1,2(s+1)/3]\subset W_1$;
                   \item[(b)] the families with $i\in W_1$, $i>2(s+1)/3$, satisfy $|\ff_i(s+1)\cap \M|\ge \frac{s+1}3$;
                   \item[(c)] the families with $i\in W_2$ satisfy  $|\ff_i((s+1)/3)\cap \M|\ge s+1$.
                 \end{itemize}
Put $u = |W_1|$ and note that we might have $u=0$. We also have $|W_2|\le \frac{s+1}3$. Note also that $|\ff_i(s+1)\cap \M|\ge (s+1)/3$ for each $i\in W_1$. The reason for the condition (a) above is to make sure that this interval is a part of $W_1$, in case both inequalities displayed above are valid for the respective family.
\begin{defn}\label{defn1} For each $i\in [s_1]$  define $m_i\in [s+1]$ to be the smallest index $j$ such that $|\ff_i(j)\cap \M|\le s+1-j$. If $|\ff_i(j)\cap \M|> s+1-j$  for all $j$ then put $m_i = \infty$.
 \end{defn}

By the definition of $m_i$, Lemma~\ref{lemrain5} implies that $m_i>\frac{5(s+1)}{6}$ for each $i$.

Rearrange the families so that the following holds.
\begin{itemize}
  \item[(R3)] We have $m_i\le m_{i'}$ for $i>i'$, where $i,i'\in [s_1]$.
  \item[(R4)] This condition is only applied if $u = 0$ and $|\ff_i(s+1)\cap \M| =\emptyset$ for each $i\in [s_1]$. (Note that in this case $s_1 = 2(s+1)/3$.)  We have   $|\ff_{s+1}(s+1)\cap \M|\ne \emptyset$. (Note that this is possible since the event $\EE_2$ holds.) 

   \end{itemize}

Let $W_1 := \{w_1,\ldots, w_u\}$, where $w_1>w_2>\ldots>w_u$, and $W_2:= \{v_1,\ldots, v_{s+1-s_1-u}\}$. 

Next, we try to construct a particular rainbow matching inside $\M$ and, by failing to do so, derive some properties of the families $\ff_1,\ldots, \ff_{s_1}$. We employ the following procedure. 
\begin{enumerate}
  \item[(1)] for each $i=1,\ldots,u$ include in the candidate rainbow matching a set from $\M\cap \ff_{w_i}(s+2-i)$ that was not used before. If $i\le \frac{s+1}3$ such a set is possible to choose since $$i\le \frac{s+1}3\le |\M\cap \ff_{w_i}(s+1)|\le |\M\cap \ff_{w_i}(s+2-i)|.$$
      If $i>\frac {s+1}3$ then $w_i\in [2(s+1)/3]$ and the validity of $\EE_1$ implies
      $|\ff_{w_i}(s+2-i)\cap \M|\ge t\alpha_{w_i,s+2-i}-\gamma \ge i$, where the last inequality 
      is due to Lemma~\ref{lemrain5} (ii).
Denote this part of the matching $\mathcal R_1$.
If $W_1 = \emptyset$ then we skip this step.
\item[(1')] If the rule (R3) of ordering the families was applied (and thus $W_1 = \emptyset$), then let $\mathcal R_1$ consist of one set, taken from $\ff_{s+1}(s+1)\cap M$. Otherwise, put $\mathcal R_1 = \emptyset$ and skip this step.\footnote{Note that both (1) and (1') are skipped if $W_1 = \emptyset$, but (R3) is not applied. Thus, $|\mathcal R|\in \{1,u\}.$}
  \item[(2)] Next, attempt to find a rainbow matching $\mathcal R_2$ that is disjoint with $\mathcal R_1$ and that covers the first $s_1$ families. Put $r = |\mathcal R_1|$ and for each $i= 1,\ldots, s_1$  try to take a set from $\ff_{i}(s+2-r-i)\cap \M$ that was not used before. 
  \item[(2)]   If the previous step was successful, complete the rainbow matching with the part $\mathcal R_3$ that is disjoint from $\mathcal R_1\cup \mathcal R_2$ and that for each $j=1,\ldots, s+1-s_1-r$ includes a set from $\ff_{v_j}(j)$ that was not taken before. Such sets are always possible to choose since
      $$|\ff_{v_j}(j)\cap \M|\ge |\ff_{v_j}((s+1)/3)\cap \M|\ge s+1.$$
        If step (1') was applied then we exclude the family $\ff_{s+1}$ from this step since one set from this family was already included in the matching.
\end{enumerate}
 As we have already pointed out, steps (1) and (3) cannot fail by our assumptions on the respective families. Since we cannot find a rainbow matching, step (2) must fail. Moreover, it must have failed for $1\le r+i-1< \frac{s+1}6$, where the right inequality is due to Lemma~\ref{lemrain5} (ii) and the validity of $\EE_1$, and the left inequality is due to $\EE_2$. The right inequality, in particular, means that $u\le r<\frac{s+1}{6}$, and thus, using property (a) of the definition of $W_1$, we get $s_1>\frac{2(s+1)}3-\frac{s+1}{6} = \frac {s+1}2$. 



Let us focus on the families $\ff_i(s+2-r-i)$, $i=1,\ldots, s_1.$ Recall Definition~\ref{defn1}. 
We must have $$|\ff_i(s+2-r-i)\cap \M|\le i+r-1$$ for some $i$, where  $1\le i+r-1<\frac{s+1}{6}$. (This follows from the same inequality two paragraphs above.)  Let $R$ be the smallest such $i$. Note that $R<\frac{s+1}{6}-r+1$ and that $$m_R\le s+2-r-R\le s,$$
where the last inequality is due to the fact that $R+r\ge 2$.


 \subsection{Concluding the proof: averaging over $\M$.} Our ultimate goal is to bound the sum $\sum_{i=1}^{s_1} \big(|\ff_i(\emptyset)|+\sum_{j=1}^{s+1}|\ff_i(j)|\big)$. As we have shown, this sum is at most $\sum_{i=1}^{s_1}\big((3s+3)|\ff_i(s+1)|+\sum_{j=1}^{s}|\ff_i(j)|\big)$. More precisely, define the random variable $$\xi:=\sum_{i=1}^{s_1}\Big((3s+3)|\ff_i(s+1)\cap \M|+\sum_{j=1}^{s}|\ff_i(j)\cap \M|\Big),$$ depending on $\M$. We will bound the expectation $\E\xi$, where the expectation is of course taken w.r.t. a $t$-matching $\M$ chosen uniformly at random, and show that  $\E\xi< s_1st$. This will imply that for at least one of $i\in[s_1]$ the inequality \eqref{rain5} fails and thus $|\ff_i|< |\aaa|$ (cf. \eqref{rain000}). Next, we provide a bound for $\E\xi$.

In case when the event $\EE_2$ fails we have $\ff_i(s+1) = \emptyset$ for all $i$, and we can write
$$\sum_{i=1}^{s_1}\Big(3(s+1)|\ff_i(s+1)\cap \M|+\sum_{j=1}^s|\ff_i(j)\cap \M|\Big)= \sum_{i=1}^{s_1}\sum_{j=1}^s|\ff_i(j)\cap \M| \le s_1 st.$$
In other words, \begin{equation}\label{rain61}
                  \E[\xi\mid \overline \EE_2] \le s_1st.
                \end{equation}
In case when the event $\EE_2$ holds but $\EE_1$ fails for $\M$, we give a trivial bound
$$\sum_{i=1}^{s_1}\Big((3s+3)|\ff_i(s+1)\cap \M|+\sum_{j=1}^{s}|\ff_i(j)\cap \M|\Big)\le s_1(4s+3)t.$$
Recall that $\Pr[\overline \EE_1]\le s^{-10}$. Thus, we have \begin{equation}\label{rain62}
                \E[\xi\mid \EE_2\cap \overline \EE_1]\cdot \Pr[\EE_2\cap \overline \EE_1]\le s_1(4s+3)t\cdot \Pr[\overline \EE_1]\le s^{-6}.
              \end{equation}
Here we use a simple bound $s_1(4s+3)t<3s^2t<s^4$, which is valid for any $s\ge 40$.

 Finally, assume that $\EE_1\cap \EE_2$ holds for $\M$. For a moment, we fix $\M$ and use notation from the previous subsection.
For $i\in [R-1]$ we use the fact that $\EE_1$ holds, and thus $|\ff_i(s+1)\cap \M|\le t\alpha_{i,s+1}+\gamma \le \frac{s+1}6+2\gamma<\frac{s+1}3$. (We used the definition of $s_1$ in the penultimate inequality and \eqref{rain210} in the last inequality.) Given this, we apply the trivial bound \begin{align*}(3s+3)|\ff_i(s+1)\cap \M|+\sum_{j=1}^s|\ff_i(j)\cap \M|\le& (3s+3)\frac{s+1}3+st\\ =& st+(s+1)^2.\end{align*}


 For $i\in [R,s_1]$ we use 
 the fact that $m_i\le m_R\le s+2-r-R$ for such $i$.
 \begin{align*}(3s+3)|\ff_i(s+1)\cap \M|+\sum_{j=1}^s|\ff_i(j)\cap \M|\le& (3s+3+s+1-m_i)(s+1-m_i)+(m_i-1)t\\
 =& st- (s+1-m_i)(t-4s-4+m_i)\\
 \le &st- (s+1-m_i)(t-4s-4)\\
 \le &st- (r+R-1)(t-4s-4),
\end{align*}
  provided $t\ge 4s+4$.

 Recall that $R< \frac{s+1}6-r+1$ and $s_1\ge \frac{2(s+1)}3-u$. Also, $r\ge u$.  Therefore,  $s_1-R+1\ge(\frac23-\frac 1{6})(s+1)=\frac{s+1}2$. Summing the bounds we obtained above and using this inequality in the third line below, we get the following inequality for the value of $\xi$ for each such $\M$.
\begin{align*}\xi \le& (R-1)\big(st+(s+1)^2\big)+(s_1-R+1)\big(st-(R+r-1)(t-5s-5)\big)\\
\le& s_1st+(R-1)(s+1)^2-(s_1-R+1)(R+r-1)(t-5s-5)\\
\le& s_1st+(R-1)(s+1)^2-\frac{s+1}2(R+r-1)(t-5s-5)\\
=& s_1st-(R-1)\frac{s+1}2\big((t-5s-5)-2(s+1)\big)-r\frac {s+1}2(t-5s-5)\\
\le& s_1st-(R-1)\frac{s+1}2 \cdot 2-r\frac {s+1}2\cdot 2\\
= &s_1st-(R+r-1)(s+1)\\
\le &s_1st-(s+1).
\end{align*}
The penultimate inequality is due to inequality \eqref{rain-1} 
and the last inequality is due to the fact that $R+r\ge 2$.

Finally, using the last displayed chain of inequalities and \eqref{rain61}, \eqref{rain62}, we have

\begin{align*}
  \E\xi\ =&\ \E[\xi\mid \overline\EE_2]\Pr[\overline\EE_2]+ \E[\xi\mid \overline \EE_1\cap \EE_2]\Pr[\overline \EE_1\cap \EE_2] + \E[\xi\mid \EE_1\cap \EE_2]\Pr[\EE_1\cap \EE_2]  \\
   \le&\ s_1st\cdot \Pr[\overline\EE_2] +s^{-6} +(s_1st-(s+1))\Pr[\EE_1\cap \EE_2]\\
  \le&\ s_1st +s^{-6} -(s+1)\big(\Pr[\EE_1]- \Pr[\overline \EE_2]\big) \\
  \le&\ s_1st +s^{-6} -(s+1) (s^{-4}-s^{-10})\\
  <&\ s_1st.
\end{align*}
Thus, for one of $i\in[s_1]$ the inequality \eqref{rain5} fails, which implies $|\ff_i|<|\aaa|$. This concludes the proof of the theorem.

\section{Concluding remarks}\label{sec4}
   Let us denote $\partial^b\ff$ to be the $b$-shadow, i.e., a collection of all $(k-b)$-sets that are contained in some set from $\ff$. The following is the main result of the paper \cite{F5} due to Peter Frankl. We state it here in a different form.
\begin{thm}[\cite{F5}]\label{thmfrfr} Fix some positive integers $n,k,b$, such that $b\le k$, and  $\alpha_b<\ldots<\alpha_k$. Consider a  family $\ff\subset{[n]\choose k}$ such that for any set $F\in \ff$ there is $i\in [b,k]$ such that $|F\cap [\alpha_i]|\ge i$. Then we have
$$|\partial^b \ff|\ge \min_{i\in [b,k]} \frac{{\alpha_i\choose i-b}}{{\alpha_i\choose i}}|\ff|.$$
\end{thm}
For convenience, let us denote the quantity in front of $|\ff|$ in the right hand side by $\beta$. We can prove this inequality in a bit simpler way using the same argument as in Lemma~\ref{lemrain4}. For completeness, let us present the proof.
\begin{proof} The proof is by induction on $n,k$.  More precisely, we use the statement for pairs $(n-1,k)$ and $(n-1,k-1)$ in order to deduce it for the pair $(n,k)$. The base cases are $k=b$, in which case we simply have $1$ set in the shadow and at most ${\alpha_b\choose b}$ sets in $\ff$, and $n = \alpha_k$, in which case any collection of sets in ${[n]\choose k}$ satisfy the condition. We then simply use that $|\partial^b\ff|\ge\frac{{n\choose k-b}}{{n\choose k}}|\ff|,$ which is valid by a simple double counting.

Let us justify the induction step.  We put $\ff':=\{F\setminus \{n\}: n\in F, F\in \ff\}$ and $\ff'':=\{F: n\notin F, F\in \ff\}$.  We have $|\ff| = |\ff'|+|\ff''|$. Moreover, $\partial^b \ff\supset \G\cup \partial^b \ff''$, where $\G:= \{\{n\}\cup A: A\in \partial^b \ff'\}$ and, clearly, $\G\cap \partial\ff'' = \emptyset$ (all sets from the first family contain $n$, while the sets from the second family do not). Note that the inductive hypothesis applies to both $\ff'$ and $\ff''$. In the first case, this is due to the fact that $n>\alpha_k$ and thus for any $A\in \ff'$ the condition on the set $\{n\}\cup A$ must be satisfied for some $i\le k-1$ (so we are actually taking minimum over fewer terms). In the second case, this is simply because $\ff''\subset \ff$. By induction, $|\G| = |\partial^b\ff'|\ge \beta|\ff'|$ and $|\partial^b \ff''|\ge \beta|\ff''|$. Therefore, $|\partial^b\ff|\ge |\G|+|\partial^b\ff''|\ge \beta|\ff'|+\beta|\ff''| =\beta|\ff|$.
\end{proof}

Returning to the main topic of the paper, we note that it is not difficult to modify the proof of Theorem~\ref{thmrain} so that it gives a stability result. However, it will only work for shifted families. Most of the proof actually does not require the families to be shifted. The main obstacle is the inclusion $\ff(\{i\},[s+1])\supset \ff(\{s+1\},[s+1])\supset\partial \ff_i(\emptyset,[s+1])$, which is valid for shifted families only.

\section{Acknowledgements} We thank the referees for carefully reading the text and pointing out several problems with the proofs and presentation. 

\end{document}